\documentclass{amsart}

\usepackage[colorlinks=true, linkcolor=blue, citecolor=blue, urlcolor=blue]{hyperref}
\usepackage{enumitem}
\usepackage{amsmath}
\usepackage{amsthm}

\usepackage{amsfonts}
\usepackage{amssymb}
\usepackage{mathrsfs}
\usepackage{mathtools}
\usepackage{bbm}
\usepackage{bm}
\usepackage{stmaryrd}
\usepackage{dutchcal}
       
\usepackage[normalem]{ulem}
\usepackage{wasysym}

\usepackage{graphicx}

\usepackage[left=3.2cm,right=3.2cm, top=2.3cm,bottom=2.3cm]{geometry}

\usepackage{tikz-cd} 
\usepackage[all,cmtip]{xy}	
\xyoption{arrow}

\usepackage{xcolor,etoolbox}
\usepackage{enumitem}
\usepackage{adjustbox}

\usepackage{empheq}

\usepackage{array}
\usepackage{multirow}

\usepackage{yhmath}

\usepackage[capitalise]{cleveref}

\theoremstyle{plain}

\newtheorem{lem}{Lemma}[subsection]
\newtheorem{prop}[lem]{Proposition}

\newtheorem{thm}[lem]{Theorem}
\newtheorem{cor}[lem]{Corollary}

\newcommand{\newheaderthm}[1]{
\newtheorem*{__headertheorem__#1}{\cref{#1}}
}

\newcommand{\startheaderthm}[1]{
\begin{__headertheorem__#1}
}
\newcommand{\finishheaderthm}[1]{
\end{__headertheorem__#1}
}

\newcommand{\headerthmfooter}{}

\newenvironment{headerthm}[1]
{
\newheaderthm{#1}
\startheaderthm{#1}
\renewcommand{\headerthmfooter}{\finishheaderthm{#1}}
}
{
\headerthmfooter
}

\newtheorem*{thm*}{Theorem}

\newtheorem*{lem*}{Lemma}
\newtheorem*{prop*}{Proposition}
\newtheorem*{cor*}{Corollary}

\theoremstyle{definition}
\newtheorem{defn}[lem]{Definition}

\newtheorem*{remark*}{Remark}
\newtheorem{remark}[lem]{Remark}

\newtheorem{ex}[lem]{Example}

\newtheorem*{question*}{Question}

\newcommand{\ZZ}{\mathbb{Z}}
\newcommand{\QQ}{\mathbb{Q}}

\newcommand{\NN}{\mathbb{N}}
\newcommand{\CC}{\mathbb{C}}

\newcommand{\UV}{\mathscr{U}}

\newcommand{\PhiL}{\Phi^\mathsf{L}}
\newcommand{\PhiR}{\Phi^\mathsf{R}}

\newcommand{\cd}[1]{\mbox{}^\theta{#1}^\vee}
\newcommand{\cdd}[1]{\mbox{}^{\theta\theta}{#1}^{\vee\vee}}

\makeatletter
\newcommand{\ostar}{\mathbin{\mathpalette\make@circled\star}}

\newcommand{\make@circled}[2]{%
  \ooalign{$\m@th#1\smallbigcirc{#1}$\cr\hidewidth$\m@th#1#2$\hidewidth\cr}%
}

\newcommand{\smallbigcirc}[1]{%
  \vcenter{\hbox{\scalebox{0.7}{$\m@th#1\bigcirc$}}}%
}
\makeatother

\newcommand{\Hom}{\mathrm{Hom}}

\newcommand{\res}{\mathrm{res}}

\newcommand{\id}{\mathrm{id}}

\newcommand{\Ac}{\mathfrak{A}}
\newcommand{\Bc}{\mathcal{Z}}

\newcommand{\Aa}{{\mathsf{A}}}
\newcommand{\Ba}{{\mathsf{Z}}}
\newcommand{\Pa}{{\mathsf{P}}}
\newcommand{\Ma}{{\mathsf{S}}}
\newcommand{\Maa}{{\mathsf{M}}}
\newcommand{\Ia}{{\mathsf{I}}}
\newcommand{\xx}{\mathfrak{a}}
\newcommand{\yy}{\mathfrak{b}}

\newcommand{\ee}{\mathfrak{e}}

\newcommand{\Mat}{\mathrm{Mat}}

\newcommand{\prt}{\mathcal{p}}
\newcommand{\PRT}{\mathcal{P}}
\newcommand{\hc}{\mathfrak{h}}
\newcommand{\whc}{\widehat{\hc}}
\newcommand{\kk}{\mathbbm{k}}
\newcommand{\vsigma}{\vec{\sigma}}
\newcommand{\vtau}{\vec{\tau}}

\newcommand{\MAP}{\operatorname{\begin{tikzpicture}
\pgfmathsetmacro{\ct}{0.11} 
\pgfmathsetmacro{\cc}{\ct*sin(18)/sin(126)} 
\draw[thin] (0,0)
    +(90-0*36:\ct) coordinate(T1)
    foreach[evaluate=\x as \nc using int((\x+1)/2),   
            evaluate=\x as \nt using int((\x+1)/2+1)] 
        \x in {1,3,...,9}{
        -- +(90-\x*36:\cc) coordinate(C\nc) -- +({90-(\x+1)*36}:\ct) coordinate(T\nt)}
    -- cycle;
\end{tikzpicture}}}

\usepackage{stackengine} 
\newcommand{\starout}{\stackMath\mathbin{\stackinset{c}{0ex}{c}{0ex}{\star}{\square}}}

\newcommand{\Otimes}[1]{\, \underset{\scriptscriptstyle{#1}}{\otimes} \,}

\newcommand{\til}[1]{\widetilde{#1}}

\newcommand{\End}{\operatorname{End}}

\newcommand{\im}{\operatorname{im}}

\newcommand{\D}[1]{{\psi}^\vee_{#1}}
\newcommand{\DD}[1]{{\psi}_{#1}}

\makeatletter
\newcommand\opteq[1]{\mathrel{\mathpalette\opt@eq{#1}}}
\newcommand{\opt@eq}[2]{%
  \begingroup
  \sbox\z@{$#1#2$}%
  \sbox\tw@{\resizebox{!}{.5\ht\z@}{$\m@th#1($}}%
  \nonscript\hskip-\wd\tw@
  \mkern1mu
  \raisebox{-.35\ht\z@}[0pt][0pt]{\resizebox{!}{.5\ht\z@}{$\m@th#1($}}%
  \mkern-1mu
  {#2}%
  \mkern-1mu
  \raisebox{-.35\ht\z@}[0pt][0pt]{\resizebox{!}{.5\ht\z@}{$\m@th#1)$}}%
  \mkern1mu
  \nonscript\hskip-\wd\tw@
  \endgroup
}
\makeatother

\usepackage[textsize=footnotesize, textwidth=25mm, color=green!40]{todonotes}

\begin{document}

\title[]{Morita equivalences for Zhu's algebra
}

\subjclass[2020]{17B69 (primary), 81R10, 16D90 (secondary)} \keywords{vertex algebras, Zhu algebra, Morita equivalence, rationality}

\begin{abstract} Through the introduction of new ideals, and with the assistance of the $d$-th mode transition algebras $\Ac_d$, for $d\in \NN$, we show how Zhu's associative algebra $\Aa$, conventionally valued for tracking information about the degree $0$ part of an $\NN$-graded module over a vertex operator algebra $V$, also contains information about components of higher degree.   As an application, equivalent conditions are given for rationality of $V$, and explicit presentations for higher-level Zhu algebras are given, including for a large class of non-rational VOAs.
 \end{abstract}

{
\author[C.~Damiolini]{Chiara Damiolini}
\address{Chiara Damiolini \newline \indent  Department of Mathematics, University of Texas at Austin,  Austin, TX  78712}
\email{chiara.damiolini@austin.utexas.edu}}
{
\author[A.~Gibney]{Angela Gibney}
\address{Angela Gibney \newline  \indent  Department of Mathematics, University of Pennsylvania,  Phil, PA 08904}
\email{agibney@math.upenn.edu}}
{
\author[D.~Krashen]{Daniel Krashen}
\address{Daniel Krashen \newline  \indent  Department of Mathematics, University of Pennsylvania,  Phil, PA 08904}
\email{dkrashen@math.upenn.edu}}

\maketitle

\vspace{-7mm}

\section*{Introduction}

Categories of modules over vertex operator algebras (VOAs) are of broad interest. VOAs with semisimple representation theory relate to rational conformal field theories and logarithmic CFTs if there are reducible and indecomposable modules \cite{CreutzigRidout}. Admissible modules (which are $\NN$-graded) over VOAs of CFT-type, together with (stable) pointed curves with coordinates define vector spaces of coinvarints and dual spaces of conformal blocks \cite{tuy, bfm, bzf, NT, DGT1}. 

The $d$-th mode transition algebra $\Ac_d$, defined in \cite[Definition 3.2.6]{DGK2}, acts on the degree $d$ component of an $\NN$-graded module, and one has $\Ac_0 = \Aa$. Here, we introduce ideals $\Ba_d$ of $\Aa$ (\cref{def:ZZ} and \cref{lem:zigzag-product}), which capture information about degree $d$ components of modules, as reflected in their degree $0$ components. The $\Ba_d$ and $\Ac_d$ are closely related, and our main result  shows that, under suitable hypotheses, they are unital algebras which are Morita equivalent:

\begin{headerthm}{thm:Acd_Bad}
  If $\Ac_d$ is \emph{strongly unital} and $\Ba_d$ is unital, there is an equivalence of categories between $\Ac_d$-modules and 
   $\Ba_d$-modules. 
\end{headerthm}

Motivated by the resemblance to the Peirce decomposition of an associative algebra, we show that $\Aa$, the $\Ac_d$, and the  $\Bc_d$ may all be considered from the broader perspective of Peirce algebras (introduced here in \cref{Peirce}), and prove  \cref{thm:Acd_Bad} in this more general context.

As in \cite{DGK2} (see also \cref{def:strong-id}),  a unit element in $\Ac_d$ is called a strong identity if it acts on the degree $d$ component of any induced $V$-module as the identity.  The $d$-th mode transition algebra $\Ac_d$ is called strongly unital if it admits a strong identity element. By \cite[Remark 3.4.6]{DGK2}, for rational $V$, the $\Ac_d$ are {strongly unital} for all $d$. 

By \cite[Theorem 2.2.3]{ZhuMod}, if $V$ is rational, then $\Aa$ is  semi-simple.  
Naturally, one would like to know what is needed, beyond the semisimplicity of $\Aa$, to guarantee that $V$ is rational.   As an application of \cref{thm:Acd_Bad}, we derive several equivalent such conditions in \cref{thm:equivalences}.  To state it, we write $\PhiL$ for the induction or generalized Verma module functor (\eqref{eq:PhiLDef} in \cref{sec:Background}). We also recall that for any left $\Aa$ module $\Maa$, the linear dual $\Maa^\vee$, a priori a right $\Aa$-module, may be given a left $\Aa$-module structure via an involution $\theta$ (see e.g.~ \cite[\S 3.4.1]{DGK2}) that we denote by $\mbox{}^\theta (\Maa^\vee)$.

\begin{headerthm}{thm:equivalences} 
The following conditions are equivalent for $V$ with semisimple Zhu algebra $\Aa$.
\begin{enumerate}
\item  For any simple $\Aa$-module $\Maa$, the generalized Verma module $\PhiL(\Maa)$ is simple and ordinary.
\item  For any simple $\Aa$-module $\Maa$, the natural map 
 $\PhiL(\cd{\Maa}) \to (\PhiL(\Maa))'$ is an isomorphism.
 \item  $\Ac_d$ is strongly unital for every $d \in \NN$.
\item  $V$ is rational.
\end{enumerate}
\end{headerthm}

In parts \ref{thm:a} and \ref{thm:b} of \cref{thm:equivalences}, the hypotheses involve all simple $\Aa$-modules at once. One may ask what can be said for a single module $\Maa$ without assuming  $\Aa$ is semisimple. In \cref{double dual simplicity} we show that if the canonical map $\DD\Maa: \PhiL({\Maa}) \to (\PhiL(\cd\Maa))'$ is an isomorphism, then it follows that $\Maa$ is finite dimensional and $\PhiL(\Maa)$ is simple (and in fact, so are $\PhiL(\cd\Maa)$ and $\PhiL(\cd\Maa)'$). If in addition, the map $\D\Maa: \PhiL(\cd{\Maa}) \to (\PhiL(\Maa))'$ is an isomorphism, 
then in \cref{prop:want modules2} we show that $\PhiL(\Maa)$ 
and $\PhiL(\cd\Maa)$ are ordinary. 
Conversely, we show also in \cref{prop:want modules2} that if both $\PhiL(\Maa)$ and 
$\PhiL(\cd\Maa)$ are simple and ordinary, then $\D\Maa$ and $\DD\Maa$ are isomorphisms.

    Determining whether or not a particular class of VOAs is rational has been a longstanding and enduring theme represented by many highly cited papers in the VOA literature, including (but not limited to), ~\cite{VirRat, DMZ, DongLiMasonRegular, AbeLatticeRationality, ArakawaInvent, AMTriplet, DongZhang2009, ArakawaWPN, OnOrbifold, DongRen}.  The statements of \cref{thm:equivalences} are related to recent work, including some results that  rely on technically demanding theory. For instance, one approach is to deduce rationality of $V$ from the properties of the tensor category of $V$-representations. Using these methods, based on \cite{HLZ}, McRae shows that if $V$ is $\NN$-graded, simple, self-contragredient,  and $C_2$-cofinite, and if $\Aa$ is semi-simple, then $V$ is rational \cite{Mcrae}.  In another recent work, \cite{MiyamotoRational}, Miyamoto has proved that if $V$ is a simple, holomorphic, and unitary vertex operator algebra of CFT-type, then $V$ is rational.

Essential to our proof, as explained in \cref{semisimple is finite}, is that for $V$ of CFT-type, if $\Aa$ is semisimple, it is also finite-dimensional. 
    That condition \ref{thm:a} implies condition \ref{thm:b} of \cref{thm:equivalences} is a generalization of \cite[Proposition 7.2.1]{NT}, while \ref{it:rtl} implies \ref{it:stun} follows from  \cite[Remark 3.4.6]{DGK2}.  
    Moreover, in \cite{NT} (\cite[page 395-396]{NT}), the authors state that they wish to find hypotheses that do not assume rationality, which nevertheless results in good conformal field theories. In particular, they introduce what they call ``Condition III'' which asserts that $V$ is $C_2$-cofinite, $\Aa$ is semisimple, and that the induction functor $\PhiL$ takes irreducible $\Aa$-modules to simple $V$-modules. But by  \cref{thm:equivalences}, \ref{thm:a},  we see that Condition III ensures the rationality of $V$.

As next explained,  we show that as consequences of these Morita equivalences, under appropriate assumptions, the higher level Zhu algebras $\Aa_d$ can be described as products of matrices of dimensions determined by the VOA itself. First, perhaps unsurprisingly, when the VOA is rational, we can use the relation between $\Ac_d$ and $\Aa_d$ described in \cite{DGK2} to show the following:

\begin{headerthm}{Cor2} Let $V$ be a rational VOA and denote by $\{S^1, \dots, S^m\}$ the set of isomorphism classes of simple admissible $V$-modules. Then \[\Aa_d\cong \prod_{j=0}^d \prod_{i=1}^m {\Mat}_{\dim(S^i_j)}(\CC),\] where $S^i_j$ is the $j$-th graded component of $S^i$, that is $S^i=\oplus_{j \in \ZZ_{\ge 0}} S^i_j$.
\end{headerthm}
Several examples where these explicit formulas are worked out are given in \cref{Sec:Examples}.

Applications of \cref{thm:Acd_Bad} extend to VOAs that are not rational and were the motivation for this work. To put the main result into context, we recall that a key step in the proof of \cref{thm:equivalences} is to show an equivalence of categories between $\Ac_d$-modules and $\Ba_d$-modules. \cref{thm:SpecialCase1} concerns a natural class of non-rational VOAs for which there is an equivalence of categories between $\Ac_d$-modules and $\Ba_d$-modules.  To describe them, we introduce a small amount of notation.

 Heuristically, when the $d$-th mode transition algebra $\Ac_d$ is strongly unital, the information in the degree $d$ part of an $\NN$-graded $V$-module $W=\oplus W_d$ (which is generated in degree zero),  is not new relative to the degree $0$ data. If $\Ac_d$ is strongly unital, and if  $W_d=0$, for some $V$-module $W$, then we find that something essential has been lost in translation, and we say  $d$ is \emph{exceptional} for $V$. We say that $d$ is \emph{non-exceptional} for $V$ if $\Ac_d$ is strongly unital, and if for all admissible $V$-modules $W$, one has $W_d \neq 0$. We have found the following, which we express using this terminology:

\begin{headerthm}{thm:SpecialCase1} Suppose $V$ is a vertex operator algebra for which $\Ac_d$ is strongly unital, and such that $\Aa$ is commutative, Noetherian, and connected. If $d$ is non-exceptional for $V$, then there is an equivalence of categories between $\Ac_d$-modules and $\Aa$-modules. If $d$ is exceptional for $V$, then $\Ac_d$ is the zero ring.
\end{headerthm}

As an application, in Corollary \ref{Cor1}, we describe a class of VOAs for which mode transition algebras $\Ac_d$ have strong units, and such that the category of $\NN$-graded V-modules behave as if they are modules over rational VOAs. This determines the structure of the higher Zhu algebras $\Aa_d(V)$ (\cite{DongLiMason:HigherZhu}), generalizing results of \cite{DGK2}.

\begin{headerthm}{Cor1}Let $V$ be a VOA such that, as a module over itself, it is simple, generated in degree zero, and has finite-dimensional graded components. Let $d$ be a positive integer and suppose that $j$ is non-exceptional for $V$ and $\Ac_j$ is strongly unital for all $j \leq d$. If $\Aa$ is commutative, connected, Noetherian, and every projective $\Aa$-module is free, then \[\Aa_d \cong \prod_{j=0}^d {\Mat}_{\dim(V_j)}(\Aa).\] 
\end{headerthm}

Corollary \ref{Cor1} applies, for instance, to Heisenberg VOAs of any rank (see \cref{ThmHeisenberg}). 

Few explicit expressions for higher-level Zhu algebras have appeared in the literature. In \cite{addabbo.barron:level2Zhu}, Addabbo and Barron compute $\Aa_2$ for the rank-1 Heisenberg VOA and conjectured that in this case, $\Aa_d$ should be a product of matrices. This conjecture was proved in \cite{DGK2}, and \cref{Cor1} is a generalization of this result. While there are special cases where it is not true that $\Aa_d$ will be a product of matrices with coefficients in $\Aa$ (see, e.g., \cite{BarronVirL1}),  if the $d$-th mode transition algebras  $\Ac_d$ have strong units  for all $d$, then we expect that  one should obtain expressions for higher level Zhu algebras  that are similar to the one given in
\cref{Cor1}.

\subsection*{Plan of the paper} We start in \cref{sec:Background} 
by reviewing the definitions and  main properties of mode transition algebras, $d$-th mode transition algebras, and strong identity elements, all of which were initially defined in \cite{DGK2}. In \cref{sec:StrongIDs}, we give several consequences for the existence of strong identity elements in $d$-th mode transition algebras. \cref{thm:Acd_Bad} and \cref{cor:connected-equiv}, the main results of this work, are proved in  \cref{Morita}.  In \cref{sec:RatIndCons}, we show that properties of the mode transition algebra $\Ac$ may be leveraged to study induced modules, properties of contragredient modules, and ultimately, to detect whether $V$ is rational  as described in \cref{thm:equivalences}. 
In  \cref{Sec:Algebraic} we prove  \cref{Cor2} and \cref{Cor1}, in which explicit expressions for higher-level Zhu algebras are given. In \cref{sec:Computations}, explicit presentations are provided for the mode transition algebras, zig-zag algebras, and higher Zhu algebras for several examples,  including the rank-$n$ Heisenberg VOA, the (non-rational) Virasoro, and lattice VOAs, which are rational.

\subsection*{Acknowledgements} Damiolini was supported by NSF DMS -- 2401420, Gibney  by NSF DMS -- 2200862, and Krashen  by NSF DMS -- 2401018.  We have benefited from conversations on this and related topics with Katrina Barron, Xu Gao, Jianqi Liu,  and Yi-Zhi Huang. We are especially grateful to Jianqi and Xu, who carefully read the manuscript and pointed out that we could remove a finiteness assumption in \cref{thm:equivalences}.  We thank the participants of the \textit{Workshop on New Directions in Conformal Field Theory} at the Fields Institute. In particular, we are grateful to Thomas Creutzig, David Ridout, and Christoph Schweigert for useful comments, questions, and advice.

\section{Mode Transition Algebras }\label{sec:Background} 

We recall here in \cref{MTADef} the definition and main properties of mode transition algebras and $d$-th mode transition sub-algebras, as initially  defined in \cite[Definition 3.2.6]{DGK2}. In \cref{strong unit section}, we define strong units. We refer to \cite{DGK2} and references therein for the background relied on here, including basics about vertex operator algebras $V$, their associated Zhu algebra $\Aa$, universal enveloping algebra $\UV$, and induced modules. We note that here every VOA is assumed to be of CFT type, which means that $V$ is $\NN=\mathbb{Z}_{\ge 0}$-graded, the graded components are finite-dimensional, and the degree zero component is one dimensional.   Throughout this work, algebras are assumed to be associative but not necessarily commutative or unital. In particular, ideals are also rings. 
\subsection{Definition and basic properties}\label{MTADef}
In \cite{DGK2}, using the induction functor, that takes a left $\Aa$-module $\Maa$ to a left $V$-module
\begin{equation}\label{eq:PhiLDef} 
    \PhiL(\Maa)  =  \dfrac{\UV}{\UV \cdot \UV_{\leq -1}} \otimes_{\UV_0} \Maa,
\end{equation}
we have explained how to associate to every VOA $V$ its \emph{mode transition algebra} 
\[\Ac=\Ac(V)=\bigoplus_{i,j \in \NN}\Ac_{i,-j},\] which is a bigraded, associative algebra. More precisely, 
\begin{equation} \label{eq:PhiL} \Ac=\PhiL(\Aa) \otimes_\Aa \PhiR(\Aa) =  \left(\dfrac{\UV}{\UV \cdot \UV_{\leq -1}} \Otimes{\UV_0} \Aa \right) \otimes_\Aa \left( \Aa \Otimes{\UV_0} \dfrac{\UV}{\UV_{\geq 1} \cdot \UV} \right)
\end{equation} where $\UV$ is the universal enveloping algebra of $V$, which is naturally graded (see \cite{DGK2}), and $\PhiL$ is the induction functor (see \eqref{eq:PhiL} or  \cite[Definition 3.1.1]{DGK2}). The multiplication on $\Ac$ is denoted $\star$ and, although not necessarily unital, it satisfies
\[ \Ac_{i,-j} \star \Ac_{k, -\ell} \subseteq \delta_{j,k} \Ac_{i,-\ell}.\]
Thus, $\Ac_d \coloneqq \Ac_{d,-d}$ is an associative algebra---called the \emph{$d$-th mode transition algebras} of $V$---and $\Ac_{i,-j}\subset \Ac$ has the structure of a $(\Ac_{i},\Ac_{j})$-bimodule. A natural way to visualize $\Ac$,  together with its bigrading and algebra structure, is by embedding it into the space of infinite dimensional matrices via the following representation
\[ \Ac= \begin{bmatrix}
    \Ac_0 & \Ac_{0,-1} & \Ac_{0,-2} & \cdots \\
    \Ac_{1,0} & \Ac_{1} & \Ac_{1,-2} & \cdots \\
    \Ac_{2,0} & \Ac_{2,-1} & \Ac_{2} & \cdots\\ 
    \vdots & \vdots & \vdots& \ddots
\end{bmatrix}, 
\]  where only finitely many terms are nonzero. In \cref{Peirce}, we isolate key structural properties of $\Ac$, and indeed, we prove \cref{thm:Acd_Bad} in the more general context of Peirce algebras, of which $\Ac$ is our primary example. 

Moreover, $\Ac$ acts on admissible $V$-modules $W=\oplus_{d\in\NN}W_d$, so that $\Ac_{i, -j} W_j \subset W_i$. In particular, one has that $\Ac_d$ acts on the $d$-th graded components $W_d$.

Note that the associativity of the multiplication may also be interpreted as defining a bilinear pairing
\begin{equation} \label{eq:pairing-ijk} \underset{i \, j \, k}{\star}  \colon \Ac_{i, -j} \otimes_{\Ac_j} \Ac_{j, -k} \longrightarrow \Ac_{i, -k},\end{equation}
which is also an $(\Ac_i,\Ac_k)$ bimodule morphism. When there is no confusion with the indices, we will denote this pairing by $\star$ and avoid further decorations.

\begin{remark} \label{rmk:rmks} We next summarize a few key features of the star product. \begin{enumerate}[label=(\ref*{rmk:rmks}.\alph*)]
    \item A a word of caution, as $\Ac_j$ need not be unital, it is not always the case that $\Ac_j \otimes_{\Ac_j} \Ac_{j, -k}$ is necessarily isomorphic to $\Ac_{j, -k}$ via $\underset{j \, j \, k}{\star}$.  
    On the other hand, if $\Ac_j$ is unital and $\Ac_{j, -k}$ a unital module via the $\star$-product, then $\underset{j \, j \, k}{\star}$ is an isomorphism. This is always true, for example, when $j = 0$.  
    \item It follows from the definition of $\star$ that one has a canonical isomorphism
 \begin{equation} \label{eq:star-iso} \underset{i\, 0 \, j}{\star} \colon  \Ac_{i, 0} \otimes_{\Ac_0} \Ac_{0, -j} \, {\overset\cong\longrightarrow} \, \Ac_{i, -j}, \qquad \xx \otimes \yy \mapsto \star(\xx \otimes \yy) = \xx \star \yy.\end{equation}
 and a pairing
\begin{equation} \label{eq:star-zero} 
\underset{0 \,j \,0}{\star} \colon  \Ac_{0, -j} \otimes_{\Ac_j} \Ac_{j, 0} \longrightarrow \Ac_0=\Aa, \qquad \xx \otimes \yy \mapsto \star(\xx \otimes \yy)=\xx \star \yy, \end{equation}
 and further, by associativity, the pairing \eqref{eq:pairing-ijk} can be written in terms of \eqref{eq:star-iso} and \eqref{eq:star-zero} as:
 \[
 \xymatrix@C=1.3cm{
 \Ac_{i, -j} \otimes_{\Ac_j} \Ac_{j, -k}  \ar[r]^{\underset{i \, j \, k}{\star}} &   \Ac_{i, -k} \\
(\Ac_{i, 0} \otimes_{\Ac_0} \Ac_{0, -j}) \otimes_{\Ac_j} (\Ac_{j, 0} \otimes_{\Ac_0} \Ac_{0, -k})  \ar[u]^{\eqref{eq:star-iso} \otimes \eqref{eq:star-iso}}_\cong \ar[d]^=&
\Ac_{i, 0} \otimes_{\Ac_0} \Ac_0 \otimes_{\Ac_0} \Ac_{0, -k}
\ar[u]^{\cong}_{\eqref{eq:star-iso}}
\\
\Ac_{i, 0} \otimes_{\Ac_0} (\Ac_{0, -j} \otimes_{\Ac_j} \Ac_{j, 0}) \otimes_{\Ac_0} \Ac_{0, -k}. \ar@/_1pc/[ru]_-{\id \otimes \eqref{eq:star-zero} \otimes \id} 
 }.
 \]
Said otherwise:
 \[  \underset{i \, j \, k}{\star} = \underset{i \, 0 \, k}{\star} \circ \left( \id \otimes \underset{0 \, j \, 0}{\star} \otimes \id \right) \circ \left( \underset{i \, 0 \, j}{\star}^{-1} \otimes \underset{j \, 0 \, k}{\star}^{-1}\right).
 \]\end{enumerate} \end{remark}

\subsection{Strong identity elements} \label{strong unit section} In \cite{DGK2} we introduced the concept of {strong identity element}. Here, we recall one of the many equivalent definitions.

\begin{defn} \label{def:strong-id} We say that an element $1_d \in \Ac_d$ is a \emph{strong identity} if it acts as the identity on the right on $\Ac_{0, -d}$ and as the identity on the left on $\Ac_{d, 0}$. That is, if for each $\xx \in \Ac_{0, -d}$ and $\yy \in \Ac_{d, 0}$, we have
\[ 1_d \star \xx = \xx \qquad \text{and} \qquad
\yy \star 1_d = \yy.\] If $\Ac_d$ admists a strong identity element, then $\Ac_d$ is said to be \emph{strongly unital}. 
\end{defn}

Said another way, $\Ac_d$ is strongly unital if and only if the subalgebra
\[\Ac_{\{0, d\}} \coloneqq \Ac_0 \oplus \Ac_{0, -d} \oplus \Ac_{d, 0} \oplus \Ac_d = 
\begin{bmatrix}
\Ac_{0, 0} &  0 & \cdots & 0 & \Ac_{0, -d} \\ 
0 && \cdots && 0\\
\vdots && \ddots &&  \vdots \\
0 && \cdots && 0\\
\Ac_{d, 0}  &  0 & \cdots & 0 & \Ac_{d, -d}  
\end{bmatrix}
\subset \Ac \]
is unital. We refer to \cite[Definition/Lemma 3.3.1]{DGK2} for other characterizations of strong identity elements.

\section{Consequences of the existence of strong identity elements}\label{sec:StrongIDs}

This section shows some of the first consequences of strong unitality. In what follows we fix a VOA $V$, we let $\Aa$ be its Zhu algebra and  $\Ac$ be its mode transition algebra. The running assumption of this section, which we will repeat for clarity, is that $\Ac_d$ is strongly unital for every $d \in \NN$.

\begin{lem} \label{lem:surj} Let $\Maa$ be an $\Aa$-module and assume that $\Ac_d$ is strongly unital for every $d \in \NN$. Then for any $V$-module $W \subset \PhiL(\Maa)$, the natural map $\PhiL(W_0) \to W$ is surjective. 
\end{lem}
\begin{proof}
Let $w \in W_d$. It is enough to show that $w$ can be written as a finite sum $\sum \alpha_i x_i$ for some $x_i \in W_0$ and $\alpha_i \in \UV_d$, where $\UV_d$ is the degree $d$ part of the universal enveloping algebra of $V$. Since $\Ac_d$ is strongly unital for every $d \in \NN$, there exists a strong identity element $\bm{1}_d$, which we can write as $\bm{1}_d = \sum \alpha_i \otimes 1 \otimes \beta_i$ as in \eqref{eq:PhiL}. Since $w \in \PhiL(\Maa)_d$, on which $\bm{1}_d$ as as the identity, then $w = \bm{1}_d w = \sum \alpha_i \beta_i w$, and setting $x_i = \beta_i w \in W_0$ gives our desired expression.
\end{proof}

Said more concretely, the previous statement tells us that, under the assumption of strong unitality, every submodule of a $V$-module induced from an $\Aa$-module is itself induced from an $\Aa$-module. In the following lemma, we show that an even stronger result holds for quotients.

\begin{lem} \label{units make quotients in degree 0}
Let $\Maa$ be an $\Aa$-module and assume that $\Ac_d$ is strongly unital for every $d \in \NN$. Suppose we have a surjective map of $V$-modules $\pi \colon  \PhiL(\Maa) \twoheadrightarrow Q$. Then the natural map $\PhiL(Q_0) \to Q$ is an isomorphism, and $\pi$ is given by applying the induction functor $\PhiL$ to $\pi_0 \colon \Maa \to Q_0$.
\end{lem}
\begin{proof} Let $W =\ker(\pi)$. As $\PhiL$ is a right exact functor and $0 \to W_0 \to \Maa \to Q_0 \to 0$ is exact we get a commutative diagram
\[\xymatrix@R=.5cm{
& & 0 \ar[d] \\
& \PhiL(W_0) \ar[r] \ar[d] & \PhiL(\Maa) \ar[d] \ar[r] & \PhiL(Q_0) \ar[d] \ar[r] & 0 \\
0 \ar[r] & W \ar[r] \ar[d] & \PhiL(\Maa) \ar[d] \ar[r] & Q \ar[r] & 0 \\
& 0  & 0
}\]
with exact rows and columns (where we use \cref{lem:surj} to ensure that $\PhiL(W_0) \to W$ is surjective). The Snake Lemma, therefore, implies the right-hand vertical map is an isomorphism, as claimed.
\end{proof}

\begin{lem} \label{generated induced}
Assume that $\Ac_d$ is strongly unital for every $d \in \NN$. If $W$ is an $\NN$-graded admissible $V$-module which is generated by its degree $0$ part $W_0$ then $W \cong \PhiL(W_0)$.
\end{lem}
\begin{proof}
As $W$ is generated by its degree $0$ part, we have a surjective map $\PhiL(W_0) \to W$ inducing an isomorphism on degree $0$. But by \cref{units make quotients in degree 0}, any such quotient must be induced by a quotient of the degree $0$ part. It follows that the above map must be an isomorphism.
\end{proof}

\begin{thm} \label{unital zero equivalence}
Assume that $\Ac_d$ is strongly unital for every $d \in \NN$. Then the functor $\PhiL$ induces an equivalence of categories between the full subcategory of $\NN$-graded admissible $V$-modules consisting of those generated by their degree $0$ parts and the category of $\Aa$-modules. 

In particular, if $W$ and $U$ are $\NN$-graded $V$-modules which are generated by their degree $0$ parts, then $\Hom_V(W, U) = \Hom_\Aa(W_0, U_0)$.
\end{thm}
\begin{proof}
We note that we have two functors \[\Omega \colon V{\mathrm{-mod}} \to \Aa{\mathrm{-mod}} \qquad \text{ and } \qquad \PhiL \colon  \Aa{\mathrm{-mod}}\to V{\mathrm{-mod}},\] where $\Omega(W)=W_0$ and such that the composition $\Omega \PhiL$ is equivalent to the identity of $\Aa$-modules (here we write $V{\mathrm{-mod}}$ to denote the category of $\NN$-graded admissible $V$-modules). Consequently, the functor $\Omega$ is automatically full and essentially surjective. We also observe that the essential image of $\PhiL$ consists exactly of those $\NN$-graded admissible modules, which are generated by their degree $0$ part, by \cref{generated induced}. 

To show that this gives an equivalence as claimed, it suffices to check that the functor $\Omega$ is faithful. For this, suppose $W$ and $U$ are generated by their degree $0$ parts (and hence induced) and suppose $f: W \to U$ with $f_0 = 0$. To show that $f = 0$, it suffices to show that $\im f = 0$. But as $\im f$ is a quotient of $W$, by \cref{units make quotients in degree 0}, the quotient map $W \to \im f$ is induced by its degree $0$ part. But $f_0 = 0$ implies that this is $0$. Hence $\im f = 0$ and $f = 0$.
\end{proof}

\begin{cor} \label{units induce simply}
Assume that $\Ac_d$ is strongly unital for every $d \in \NN$. Then for $\Maa$ an $\Aa$-module, $\PhiL(\Maa)$ is simple if and only if $\Maa$ is.
\end{cor}
\begin{proof}
As submodules and quotients of $\PhiL(\Maa)$ are in bijection with submodules and quotients of $\Maa$ by \cref{unital zero equivalence}, it follows that $\Maa$ has no proper submodules exactly when $\PhiL(\Maa)$ has no proper submodules. Hence, one is simple if and only if the other is. 
\end{proof}

\section{Morita equivalences}\label{Morita}

Here, we prove \cref{thm:Acd_Bad} and \cref{cor:connected-equiv}, the main results of this work. To begin, we consider properties of the map \eqref{eq:star-zero}, which we denote $\star$ here. 
\subsection{Peirce algebras and the zig-zag algebras}
While we are primarily interested in structures arising from the study of VOAs, the results in this section apply in a more general context. In particular, we may encapsulate the relevant structures from the previous section by introducing the following definition (motivated by the resemblance to the Peirce decomposition of an associative algebra):

\begin{defn}\label{Peirce}
A \emph{Peirce algebra} is a (not necessarily unital) algebra $\Ac$ (over a commutative ring $R$) with an $R$-module bigrading $\Ac=\bigoplus_{i,j \in \NN}\Ac_{i,-j}$ such that 
\begin{enumerate}
    \item $\Ac_{i,-j} \star \Ac_{k, -\ell} \subseteq \delta_{j,k} \Ac_{i,-\ell}$,
    \item $\Aa \coloneqq \Ac_{0, 0}$ is a unital subalgebra,
    \item the product endows $\Ac_{i, 0}$ (resp. $\Ac_{0, -j}$) with a unital right (resp. left) $\Aa$-module structure,
    \item \label{contract} the product induces an isomorphism $\Ac_{d, 0} \otimes_\Aa \Ac_{0, -d} \cong \Ac_{d, -d}$.
\end{enumerate}  
As in \cref{sec:Background}, we can define the $d$-th Peirce subalgebra as $\Ac_d \coloneqq \Ac_{d, -d}$ and, as in \cref{def:strong-id} we can define the notion of these being strongly unital.
\end{defn}

As was shown in \cref{sec:Background}, the mode transition algebra has the structure of a Peirce algebra. In the following, we fix a Peirce algebra $\Ac$.

\begin{defn}\label{def:ZZ}
    For every $d \in \NN$, we define $\Bc_d = \Bc_d(\Ac)$ to be the space
    \[\Bc_d \coloneqq \Ac_{0, -d} \otimes_{\Ac_d} \Ac_{d, 0}.\]
The image of $\Bc_d$ under $\star \colon \Bc_d \to \Aa$ will be denoted $\Ba_d = \Ba_d(\Ac)$.
\end{defn} 

Observe that the morphism $\star$ is an $\Aa$ bimodule map, thus $\Ba_d$ is a two-sided ideal of $\Aa$. The following statement shows that $\Bc_d$ is naturally an $\Aa$-algebra and $\star$ is an $\Aa$-algebra homomorphism. We refer to $\Bc_d$ as the \emph{$d$-th zig-zag algebra} of $V$. 

\begin{lem} \label{lem:zigzag-product} There is a natural map \[ {\MAP} \colon \Bc_d \otimes_{\Aa} \Bc_d \longrightarrow \Bc_d, \qquad \xx \otimes \yy \mapsto  \MAP(\xx \otimes \yy) \coloneqq \xx \MAP \yy \] which gives $\Bc_d$ the structure of an associative (not necessarily unital) $\Aa$-algebra, and such that the map $\star \colon \Bc_d \to \Aa$ is a (not necessarily unital) homomorphism of $\Aa$-algebras. Furthermore, if $\Ac_d$ is strongly unital, the map $\MAP$ is an isomorphism. 
\end{lem}

\begin{proof} We first construct the map $\MAP$. By definition we have $\Aa=\Ac_0$ and thus (using \cref{Peirce}\eqref{contract}),
\[\Bc_d \Otimes{\Ac_0} \Bc_{d} =  \left( \Ac_{0, -d} \Otimes{\Ac_d}  \Ac_{d,0}\right) \Otimes{\Ac_0} \left( \Ac_{0,-d} \Otimes{\Ac_d} \Ac_{d, 0}\right) = \Ac_{0, -d} \Otimes{\Ac_d} \Ac_d \Otimes{\Ac_d} \Ac_{d, 0}.
\]
Notice further that the map
\[ \Ac_{0, -d} \Otimes{\Ac_d} \Ac_d \Otimes{\Ac_d} \Ac_{d, 0} \to \Ac_{0,-d} \Otimes{\Ac_d} \Ac_{d,0}=\Bc_d, \quad u \otimes a \otimes v \mapsto u \star a \otimes v = u \otimes a \star v
\] is well defined and, when $\Ac_d$ is strongly unital, then it is an isomorphism. Hence we have shown that there is a natural map 
\begin{equation*} \label{eq:mult-Bc} {\MAP} \colon  \Bc_d \otimes_{\Aa} \Bc_{d} \to \Bc_d
\end{equation*} which is an isomorphism when $\Ac_d$ is unital. We now show that this defines an associative operation on $\Bc_d$, which induces an $\Aa$-algebra structure on $\Bc_d$. Explicitly, the map $\MAP$ is given by 
    \begin{equation}\label{products}
(a_1 \otimes a_2) \MAP (b_1 \otimes b_2) = a_1 \star a_2 \star b_1 \otimes b_2 = a_1 \otimes a_2 \star b_1 \star b_2.
    \end{equation} where $a_1,b_1 \in \Ac_{0,-d}$ and $a_2, b_2 \in \Ac_{d,0}$.
    We check that the operation $\MAP$ is indeed associative using the associativity of the $\star$-product:
    \begin{multline*}
    \big((a_1 \otimes a_2) \MAP (b_1 \otimes b_2)\big) \MAP (c_1 \otimes c_2) = (a_1 \otimes (a_2 \star b_1 \star b_2) )\MAP (c_1 \otimes c_2) 
    \\= a_1 \otimes \left( a_2 \star b_1 \star b_2 \star c_1 \star c_2 \right) 
    \\= (a_1 \otimes a_2) \MAP ((b_1 \star b_2 \star c_1) \otimes c_2)= (a_1 \otimes a_2) \MAP \big((b_1 \otimes b_2) \MAP (c_1 \otimes c_2)\big).
    \end{multline*}
Finally, we are left to show that the map $\star  
 \colon \Bc_d \to \Aa$ is an algebra homomorphism. This is true because, by identifying $\Aa$ with $\Ac_0$, the product on $\Aa$ is the star product.
\end{proof}

In what follows we denote the action of $\Aa$ on $\Bc_d$ by $\starout$ and use $\star$ to denote the product of $\Ac$, as well as the map $\Bc_d \to \Ba_d \subseteq \Aa$. As in \cref{lem:zigzag-product}, the multiplication in $\Bc_d$ is denoted $\MAP$. These different types of operation are involved, as shown in the following lemma.

\begin{lem} \label{lem:action-through-A}
The left (resp. right) action of $\Bc_d$ on itself factors through the natural action of $\Aa$ via the homomorphism $\Bc_d \twoheadrightarrow \Ba_d \hookrightarrow \Aa$, that is
\[ \xx \starout (\star(\yy))= \xx \MAP \yy = (\star(\xx)) \starout \yy 
\] for every $\xx, \yy \in \Bc_d$.
\end{lem}

\begin{proof} The statement follows from the chain of equalities
\[
(a_1 \otimes a_2) \MAP (b_1 \otimes b_2) = (a_1 \star a_2 \star b_1) \otimes b_2 = (a_1 \star a_2) \star b_1 \otimes b_2 = (a_1 \star a_2) \starout (b_1 \otimes b_2),
 \]  where the last equality holds because $\Aa$ acts on $\Bc_d$ via the star product.
\end{proof}

\begin{lem} \label{bad idempotent ideal}
Suppose that $\Ac_d$ is strongly unital. Then $\Ba_d$ is an idempotent (two-sided) ideal of $\Aa$.     
\end{lem}
\begin{proof}
It follows from \cref{lem:zigzag-product} that whenever $\Ac_d$ is strongly unital, then $\Bc_d \MAP \Bc_d$, simply denoted $\Bc_d^2$, coincides with $\Bc_d$ and that this is independent of whether the algebra $\Bc_d$ is unital or not. Consequently, as the map $\star \colon \Bc_d \to \Aa$ is a ring map, its image, the ideal $\Ba_d \triangleleft \Aa$, also satisfies $\Ba_d^2 = \Ba_d$.     
\end{proof}

\begin{prop} Suppose that $\Ba_d$ is unital, with identity $\epsilon_d$. Then  $\Aa = \Ba_d \times \Ba_d'$ where $\Ba_d'$ is the two-sided ideal generated by the complementary idempotent $1_\Aa - \epsilon_d$.
\end{prop}

Since $\Aa$ is unital, this follows immediately from the following standard lemma.

\begin{lem} \label{lem:idemgood+splitting} If $B$ is a unital ring, then an ideal $I \triangleleft B$ is generated by a central idempotent if and only if it is unital as a ring. In this case, one has the decomposition $B = I \times I'$ where $I'$ is the two-sided ideal of $B$ generated by the complementary idempotent $1_B -\epsilon$. 
\end{lem}

\begin{proof} If $I$ is generated by a central idempotent $\epsilon$, then necessarily $\epsilon$ is the identity element of $I$, so we only need to show the other implication. Let $\epsilon \in I$ be the identity of $I$. We claim that $\epsilon$ is central and $I = B\epsilon$. For the first, we have $B\epsilon \subset I = I \epsilon \subset B\epsilon$, and similarly $I = \epsilon B$. Let $\eta = 1_B - \epsilon \in B$ be the complementary idempotent. We then have
\[ B = B\eta \oplus B\epsilon = \eta B \oplus \epsilon B,\]
and $\epsilon B = I = B \epsilon$. If $a \in B$, we therefore can write
$a = a \epsilon + a \eta$, and as $a \epsilon \in I = \epsilon B$, it follows that $\epsilon(a \epsilon) = a \epsilon$, and therefore
\[\epsilon a = \epsilon(a \epsilon) + \epsilon a \eta = a \epsilon + \epsilon a \eta.\]
But $\epsilon a \in I = B \epsilon$ and so $\epsilon a \eta = 0$, which shows $\epsilon a = a \epsilon$ as claimed.
\end{proof}

\begin{ex} \label{ex:idemgood} It is common for rings to have the property that a central idempotent generates every idempotent ideal:
\begin{enumerate} [label=(\ref*{ex:idemgood}.\alph*)]    
\item \label{commutative connected example} If $B$ is a unital commutative and $I$ is a finitely generated idempotent ideal of $B$, then $I$ is generated by a central idempotent \cite[Cor~6.3]{MulIT}. 
\item \label{ss example} If $B$ is unital and semisimple, then it is is a product $B = B_1 \times \cdots \times B_m$, for simple rings $B$ (see \cite[Thm~3.5]{Lam:FCNR}). Therefore, every idempotent ideal is generated by a central idempotent, as we have $B_i = e_i B$ for central primitive idempotents $e_i$, and hence, all ideals are generated by sums of such central idempotents.
\item \label{rtl example} If $V$ is a VOA, then the zeroth Zhu algebra $\Aa=\Aa(V)$ is a unital ring, which is not necessarily commutative nor Noetherian. However, when $V$ is rational, then $\Aa$ is finite and semisimple. Thus, we can use \ref{ss example} to guarantee that all idempotent ideals are generated by a central idempotent.
\end{enumerate}
\end{ex}

In many contexts---see for instance \cref{ex:idemgood}---one could use \cref{lem:idemgood+splitting} to deduce that $\Ba_d$ contains a central idempotent which would act as its identity element, and hence $\Aa \cong \Ba_d \times \Ba_d'$. It, therefore, makes sense to consider the following, not uncommon occurrence. 

\begin{prop} \label{B units}
    Suppose $\Ac_d$ is strongly unital and $\Ba_d$ is generated by a central idempotent of $\Aa$ (equivalently, is unital as a ring). Then $\Bc_d \cong \Ba_d$ and thus $\Aa$ decomposes as $\Bc_d \times \Ba_d'$.
\end{prop}

\begin{proof} Let $\epsilon_d \in \Ba_d$ be an idempotent generator of $\Ba_d$. 
Given that $\epsilon_d$ is idempotent, we deduce that
\[ \epsilon_d \starout (\epsilon_d \starout \xx)= (\epsilon_d \star \epsilon_d) \starout \xx = \epsilon_d \starout \xx,
\] for every $\xx \in \Bc_d$, so that $\epsilon_d$ acts as the identity on elements of $\epsilon_d \starout \Bc_d$.  Using \cref{lem:action-through-A}, and the fact that $\epsilon_d$ generates $\Ba_d$, we deduce that 
\[\epsilon_d \starout \Bc_d = \epsilon_d \starout  (\Bc_d \MAP \Bc_d) = \epsilon_d \starout (\Ba_d \starout \Bc_d) = (\epsilon_d \star \Ba_d) \starout \Bc_d = \Ba_d \starout \Bc_d = \Bc_d \MAP \Bc_d = \Bc_d,\] 
so that $\epsilon_d$ acts as the identity on the whole $\Bc_d$. Consequently, the algebra map $\Bc_d \to \Ba_d$ is also a map of unital $\Ba_d$-bimodules.

Now, choose an element $\ee_d \in \Bc_d$ which maps to $\epsilon_d$ via $m_\star$. Using  \cref{lem:action-through-A}, for every $\xx \in \Bc_d$ one has:
\[ \ee_d \MAP \xx = \epsilon_d \starout \xx = \xx,\]
and similarly on the right, showing that $\ee_d$ is the identity of $\Bc_d$.

To show that $\Bc_d$ and $\Ba_d$ are isomorphic, we define the morphism of left $\Ba_d$-modules
\[\varepsilon_d \colon \Ba_d \to \Bc_d, \quad \beta \mapsto \beta \starout \ee_d.\] As a consequence of \cref{lem:action-through-A}, we can check that the composition
\begin{align*}
\Bc_d \longrightarrow \Ba_d &\longrightarrow \Bc_d \\
a \otimes b \mapsto a \star b &\mapsto (a \star b) \starout \ee_d = (a \otimes b) \MAP \ee_d = a \otimes b,
\end{align*}
is the identity, which shows that the surjective ring homomorphism $\star \colon \Bc_d \to \Ba_d$ is necessarily injective, thus an isomorphism. The rest of the statement follows from \cref{lem:idemgood+splitting}. \end{proof}

\subsection{Morita equivalences}
Before stating and proving \cref{thm:Acd_Bad},  we recall that if $B$ is a unital ring, then $B$-modules are assumed to be unital.

\begin{thm} \label{thm:Acd_Bad}
    Let $\Ac$ be a Pierce algebra. If $\Ac_d$ is \emph{strongly unital} and $\Ba_d = \Ba_d(\Ac)$ is unital, there is an equivalence of categories between $\Ac_d$-modules and 
   $\Ba_d$-modules. 
\end{thm}

\begin{proof}
We define these equivalences via the functors
\begin{align*}
\Ac_d\text{-mod} &\longrightarrow \Ba_d\text{-mod} \\  
W_d &\longmapsto \Ac_{0, -d} \otimes_{\Ac_d} W_d,
\end{align*}
and 
\begin{align*}
\Ba_d\text{-mod} &\longrightarrow \Ac_d\text{-mod} \\  
W_0 &\longmapsto \Ac_{d, 0} \otimes_{\Ac_0} W_0 = \Ac_{d, 0} \otimes_{\Ba_d} W_0.
\end{align*}
Note that here we are implicitly considering the decomposition $\Aa = \Ba_d \times \Ba_d'$ induced from the fact that $\Ba_d$ is an ideal generated by a central idempotent, which allows us to consider $\Ba_d$ as either a subalgebra or a quotient of $\Aa$.

We check that these are inverse equivalences:
\[
W_d \; \mapsto \; \Ac_{0, -d} \otimes_{\Ac_d} W_d \; \mapsto \; \Ac_{d, 0} \otimes_{\Ac_0} \Ac_{0, -d} \otimes_{\Ac_d} W_d = \Ac_d \otimes_{\Ac_d} W_d \cong W_d. 
\]
Note that the last isomorphism relies on $\Ac_d$ having an identity element and a unital action on $W_d$. In the other direction, we have:
\[
W_0 \; \mapsto \; \Ac_{d, 0} \otimes_{\Ba_d} W_0 \; \mapsto \; \Ac_{0, -d} \otimes_{\Ac_d} \Ac_{d, 0} \otimes_{\Ba_d} W_0 = \Bc_d \otimes_{\Ba_d} W_0 = \Ba_d \otimes_{\Ba_d} W_0 \cong W_0, 
\]
again relying on the identity element in $\Ba_d$ and its unital action on $W_0$.
\end{proof}

The following corollary is immediate.

\begin{cor}\label{cor:connected-equiv} Let $\Ac$ be a Pierce algebra. Suppose that $\Ac_d$ is strongly unital, $\Ba_d$ is unital and $\Aa$ has no nontrivial idempotent ideals. Then, we have an equivalence of categories between $\Ac_d$-modules and $\Aa$-modules.
\end{cor}

We now specialize the above results to the case where the Peirce algebra $\Ac$ is the mode transition algebra of a VOA $V$. For a given admissible $V$-module $W=\oplus_{d\in \ZZ_{\ge 0}} W_d$, one may have that $W_d= 0$ for some $d$. This happens; for instance, for the Virasoro VOA, $V_c$ is considered as a module over itself which, by construction, has $(V_c)_1=0$ (see \cref{ExVir}).
Here, we show that if there is some module over $V$ with nonzero degree $d$ part, then under natural assumptions, Zhu's algebra and the $d$-th mode transition algebras share the same representation theory. We first introduce some terminology. For a VOA $V$, we say that an $\NN$-graded $V$-module $W=\oplus_{i \in \NN} W_i$ is \emph{generated in degree $0$} if there is a surjection $\PhiL(W_0) \to W$, where the functor $\PhiL$ is given in \eqref{eq:PhiL}.

\begin{lem} \label{lem:V-has-d}
For $V$ a VOA, the following are equivalent:
\begin{enumerate}[label=(\roman*)]
\item \label{general} There exists an $\NN$-graded $V$-module $W$, generated in degree $0$ with $W_d \neq 0$,
\item \label{specific} $\PhiL(\Aa)_d \neq 0$.
\end{enumerate}
Further, if $\Ac_d = \Ac_d(V)$ is strongly unital, these conditions are equivalent to $\Ac_d \neq 0$. 
\end{lem}

Note that the $0$-ring is unital, and observe that when $\Ac_d =0$, then by the isomorphism \eqref{eq:star-iso} the module $\Ac_{0,d}$ is also zero, and thus $\Ac_d$ is necessarily strongly unital.

\begin{proof}
Since $\PhiL(\Aa)$ is an $\NN$-graded $V$-module with degree $d$ component given by $\PhiL(\Aa)_d$, then clearly \ref{specific} implies \ref{general}. 
On the other hand, if \ref{general} holds, so we have some $W$ as above with $W_d \neq 0$, we claim that for some $w \in W_0$, if we consider the submodule $\langle w \rangle$ generated by $w$, we have $\langle w \rangle_d \neq 0$ for some $w$. This is clear because $W_d = \sum_w \langle w\rangle_d \neq 0$. But now, if we choose such a $w$ and consider the map of $\Aa$-modules $\Aa \to W_0$ by $1 \mapsto w$, we see this extends to a $V$-module map $\PhiL(\Aa) \to W$ whose image is exactly $\langle w \rangle$. In particular, we get a surjection $\PhiL(\Aa)_d \to \langle w \rangle_d$, showing that $\PhiL(\Aa)_d \neq 0$ and so \ref{specific} holds.

As the identity element element of $\Ac_d$ has a nonzero action on $W_d$, it is not the zero ring when such a module $W$ exists. Conversely, from the expression $\Ac_d = \PhiL(\Aa)_d \otimes_{\Aa} \PhiR(\Aa)_{-d}$, we see that  $\Ac_d = 0$ when $\PhiL(\Aa)_d = 0$.
\end{proof}

\begin{defn}\label{Def:Vanishes}  We say that \emph{$d$ is non-exceptional for $V$} if one of the equivalent conditions of \cref{lem:V-has-d} are satisfied. Otherwise, we say that \emph{$d$ is exceptional for $V$}.
\end{defn}

\begin{thm}\label{thm:SpecialCase1} Let $V$ be a VOA such that $\Ac_d$ is strongly unital and its Zhu algebra $\Aa$ is commutative, Noetherian, and connected.
\begin{enumerate}[label=(\arabic*)]
\item \label{it.a1}  If $d$ is non-exceptional for $V$, then there is an equivalence of categories between $\Ac_d$-modules and $\Aa$-modules. 
\item If $d$ is exceptional for $V$ then $\Ac_d$ is the zero ring. \end{enumerate}
\end{thm}

\begin{proof} From \cref{lem:V-has-d} it is enough to prove \ref{it.a1} and assume that $\Ac_d$ is not the zero ring. By \cref{bad idempotent ideal}, $\Ba_d$ is an idempotent ideal of $\Ac_d$ and, as observed in Example~\ref{commutative connected example}, 
the ideal $\Ba_d$ is generated by a central idempotent. Since $\Aa$ is connected, this implies $\Ba_d = \Aa$ or $\Ba_d = 0$. In either case, we obtain an equivalence of categories between $\Ac_d$-modules and $\Ba_d$-modules from \cref{{thm:Acd_Bad}}. But since $\Ac_d \neq 0$, then necessarily also $\Ba_d \neq 0$ and thus $\Ba_d \cong \Aa$, concluding the argument. 
\end{proof}

\section{Rationality, induced modules, and contragredients}\label{sec:RatIndCons}
As we show here in this section, we can use the properties of the mode transition algebra $\Ac$ and our results on Morita equivalences to understand induced modules, properties of contragredient modules, and ultimately, to detect whether $V$ is rational  as described in \cref{thm:equivalences}. 

To begin with we prove the following result.
\begin{prop} \label{cor:SpecialCase2} For a rational VOA $V$, we have an equivalence of categories between $\Ac_d$-modules and $\Bc_d$-modules.
\end{prop}

\begin{proof}
By \cite[Remark 3.4.6]{DGK2}, for rational vertex operator algebras $V$ of CFT-type, $\Ac_d$ is strongly unital, and so by \cref{bad idempotent ideal}, $\Ba_d$ is an idempotent ideal of $\Ac_d$. By Example~\ref{rtl example}, 
the ideal $\Ba_d$ is unital, and the result follows from \cref{{thm:Acd_Bad}}.
\end{proof}

\begin{defn}\label{def:psiMap}
    For an $\Aa$-module $\Maa$, as $\cd\Maa$ can be identified with both the degree $0$ components of $\PhiL(\cd\Maa)$ and $\PhiL(\Maa)'$, the universal property of induced modules (\cite[Proposition B.1.4]{DGK2}) gives a canonical map $\D{\Maa}: \PhiL(\cd\Maa) \to \PhiL(\Maa)'$. 
    Similarly, as we may identify the degree $0$ part of $\PhiL(\cd\Maa)'$ with 
    $\cdd\Maa$, the canonical map $\Maa \to \cdd\Maa$ induces a canonical map $\DD{\Maa} \colon \PhiL(\Maa) \to \PhiL(\cd\Maa)'$.
\end{defn}

\begin{remark}\label{rk:Maps}
We observe the following properties of the maps $\D{\Maa}$ and $\DD{\Maa}$.

\begin{enumerate}[label=(\ref*{rk:Maps}.\alph*)]    
    \item If $\Maa$ is an infinite-dimensional,  while $\D{\Maa}$ may be an isomorphism,  $\DD{\Maa}$ may never be so.

\item \label{simple dual remark} As   $\DD\Maa$ and $\D\Maa$ are nonzero maps, $\DD\Maa$ will be an isomorphism when both $\PhiL(\Maa)$ and $\PhiL(\cd\Maa)'$ are simple, and $\D\Maa$ will be an isomorphism when both $\PhiL(\cd\Maa)$ and $\PhiL(\Maa)'$ are simple. 

\item \label{finite dimensional duals} If $\Maa$ is a finite dimensional $\Aa$-module, then we have a canonical isomorphism $\Maa \overset \sim\to \cdd\Maa$. In particular, if we precompose $\DD{\cd\Maa}$ with the identification $\PhiL(\Maa) \to \PhiL(\cdd\Maa)$ induced by this isomorphism in degree $0$, we see there is a natural identification $\DD{\cd\Maa} = \D{\Maa}$. Applying this reasoning to $\cd\Maa$, we obtain that $\D{\cd\Maa} = \DD{\Maa}$.
\end{enumerate}
\end{remark}

\begin{remark}\label{countable dimension}
We note that (see \cite{DongMasonRadical}), if $V$ is a VOA of CFT-type and $W$ is a countably generated weak $V$-module, then $W$ has countable dimension. This follows from the fact that there are a countable number of finite strings of operators of the form $a_{(m)}$ where $a \in V$ and $m \in \ZZ$. Consequently, as $\PhiL(\Aa)$ is a finitely generated $V$-module, $\Aa$ always has countable dimension.
\end{remark}

\begin{remark} \label{semisimple is finite}
One can show from \cref{countable dimension} that if $\Aa$ is semisimple, it is also finite-dimensional. To see this, we note that by Wedderburn-Artin, $\Aa$ is a finite product of matrix algebras with entries in division algebras. It suffices to show that any such division algebra $D$ containing $\CC$ must either be $\CC$ or have uncountable dimension over $\CC$. But for this, suppose $d \in D \setminus \CC$. As $\CC$ is algebraically closed, and as rational expressions in $d$ with coefficients in $\CC$ must commute with each other, it follows that $\CC(d) \subset D$ is isomorphic to the field of rational functions in one variable. But then we see that the expressions $1/(d - \lambda)$ are $\CC$-linearly independent as $\lambda$ varies, showing the space has uncountable dimension
\end{remark}
\begin{lem} \label{double dual simplicity}
Let $V$ be a VOA with Zhu algebra $\Aa$, and suppose that $\Maa$ is a simple $\Aa$-module such that the canonical map $\DD{\Maa}$ is an isomorphism. Then $\Maa$ is finite-dimensional, and the $V$-modules $\PhiL(\Maa)$, $\PhiL(\cd{\Maa})$, and $\PhiL(\cd{\Maa})'$  are all simple.
\end{lem}

\begin{proof}
First note that as the map $\DD\Maa$ on the degree $0$ parts is, on the level of vector spaces, the canonical map $\Maa \to \Maa^{\vee\vee}$ to the double dual, and the fact that it is an isomorphism implies that $\Maa$ is finite-dimensional. We also note that the functor $() \mapsto \cd{()}$ is an anti-equivalence on the category of finite dimensional $\Aa$ modules, taking subobjects to quotients. It then follows that $\cd\Maa$ is also a simple and finite-dimensional $\Aa$-module.

By \cite{ZhuMod}, we have short exact sequences
\[0 \to R \to \PhiL(\Maa) \to S \to 0,\]
\[0 \to \til R \to \PhiL(\cd\Maa) \to \til S \to 0,\]
  where $S$ and $\til S$ are simple $V$-modules, and the map on degree $0$ parts $\Maa \to S_0$ and $\cd\Maa \to \til S_0$ are isomorphisms, and consequently $R_0 = \til R_0 = 0$. Taking the contragradient modules in the second sequence and identifying $\PhiL(\cd\Maa)'_0 = \Maa$, we get a diagram
  \[\xymatrix{
  0 \ar[r] & R \ar[r] & \PhiL(\Maa) \ar[d]_{\DD{\Maa}} \ar[r] & S \ar[r] & 0 \\
  0 \ar[r] & \til S' \ar[r] & \PhiL(\cd\Maa)' \ar[r] & \til R' \ar[r] & 0
  }\]
The $V$-module $\PhiL({\Maa})$ is generated by its elements of degree $0$, and  $\til R_0 = 0 =  \til R'_0$. It then follows that composite map $\PhiL({\Maa}) \to \til R'$ is zero, and hence the image of the map $\psi_{{\Maa}}$ lies entirely in $\til S'$. But as $\psi_{{\Maa}}$ is an isomorphism, it must, in particular, be surjective, which shows $\til R' = 0$ and $\til S' \cong \PhiL(\Maa)$. But therefore $\til R = 0$ and $\PhiL(\cd\Maa) \cong \til S$ is simple. And again, as $\PhiL(\Maa) \cong \PhiL(\cd\Maa)'$, we have that $\PhiL(\Maa)$ and $\PhiL(\cd\Maa)'$ are also both simple, as claimed.
  \end{proof}
  \begin{lem}\label{lem:Ordinary}
 Let $V$ be a VOA with Zhu algebra $\Aa$, and suppose that $\Maa$ is a simple $\Aa$-module such that the canonical maps $\DD{\Maa}$ and $\DD{\cd\Maa}$ are both isomorphisms. Then $\PhiL(\Maa)$ is ordinary.
  \end{lem}

\begin{proof}
By \cref{double dual simplicity}, $\PhiL(\Maa)$ and $\PhiL(\cd\Maa)$ are simple and $\Maa$ and $\cd\Maa$ are finite dimensional.
  Consider the natural map $\PhiL(\Maa) \to \PhiL(\Maa)''$. As $\Maa$ is finite-dimensional, the natural map $\Maa \to \cdd{\Maa}$ is an isomorphism, and we may consider the map obtained by composing:
\begin{equation}\label{double dual mapping}
\PhiL(\Maa) \to \Phi\left(\cdd{\Maa}\right)''.  
\end{equation}
Using the fact that morphisms from induced modules are identified with maps on the degree $0$ parts as $\Aa$-modules, we note that we can factor this map as
\[\xymatrix{
\PhiL(\Maa) \ar[rr] \ar[rd]_{\psi_\Maa} & &  \PhiL\left(\cdd{\Maa}\right)'' \\
& \PhiL(\cd\Maa)' \ar[ru]_{(\psi_{\cd\Maa})'}
}\]
  By hypothesis, the canonical maps $\DD{\Maa}$ and $\DD{\cd\Maa}$ are isomorphisms, and we can also conclude that  $(\DD{\cd\Maa})'$ is an isomorphism. So, by the diagram,  the map of \eqref{double dual mapping} is an isomorphism. But this is the map between $\PhiL(\Maa)_d$ and its double dual in each degree. As these are isomorphisms, each space must be finite-dimensional; hence, $\PhiL(\Maa)$ is ordinary, as claimed. 
\end{proof}
The following result is a straightforward consequence of \cref{double dual simplicity} and \cref{lem:Ordinary}.

\begin{prop}\label{prop:want modules2}
    Let $V$ be a VOA, and $\Maa$ a simple $\Aa$-module.  The following are equivalent:
    \begin{enumerate}[label=(\alph*)]    
        \item \label{it:moda} $\DD{\Maa}$  and $\D{\Maa}$ are isomorphisms;
        \item \label{it:modb} $\PhiL(\Maa)$ and $\PhiL(\cd\Maa)$ are simple ordinary modules and $\Maa$ is finite-dimensional.
    \end{enumerate}
    \end{prop}
  \begin{proof}
        That \ref{it:moda} implies \ref{it:modb} follows from \cref{double dual simplicity} and \cref{lem:Ordinary}. In particular, by \cref{double dual simplicity}, the assumption that $\DD{\Maa}$ is an isomorphism gives that $\Maa$ is finite-dimensional and the $V$-modules $\PhiL(\Maa)$ and $\PhiL(\cd{\Maa})$   are  simple. Since $\Maa$ is finite-dimensional, $\D{\Maa}=\DD{\cd\Maa}$.  Then again, by \cref{double dual simplicity}, the assumption $\DD{\cd\Maa}$ is an isomorphism gives that $\cd\Maa$ is finite-dimensional. Taken together, by \cref{lem:Ordinary}, since  $\Maa$ and $\cd\Maa$ are finite dimensional, and $\DD{\cd\Maa}$ and $\DD{\Maa}$  are isomorphisms,  we conclude that $\PhiL(\Maa)$ and $\PhiL(\cd{\Maa})$ are ordinary.

        That \ref{it:modb} implies \ref{it:moda} follows from  Remark~\ref{simple dual remark}.
    \end{proof}  

\begin{prop} \label{prop:we're out here making units}
Let $V$ be a VOA whose Zhu algebra $\Aa$ is semisimple. Suppose that for any simple $\Aa$-module $\Maa$, the natural map  $\D\Maa \colon \PhiL(\cd{\Maa}) \to (\PhiL(\Maa))'$ is an isomorphism. Then $\Ac_d$ is strongly unital for every $d \in \NN$.
\end{prop}

\begin{proof}
Recall that since $\Aa$ is semisimple, one has that $\Aa$ is finite-dimensional (see \cref{semisimple is finite}). In particular, so is every simple $\Aa$-module. Hence, by Remark~\ref{finite dimensional duals}, it follows that we may identify $\DD\Maa = \D{\cd\Maa}$ for any simple $\Aa$-module $\Maa$, and in particular, we have that also $\D{\Maa}$ is an isomorphism. We can, therefore, apply \cref{lem:Ordinary} and obtain that $\PhiL(\Maa)$ is ordinary for any $\Maa$ simple.

 We will now explicitly construct the strong identity element of $\Ac_d$, generalizing the argument made in \cite[Remark 3.4.6]{DGK2} made for $V$ rational.  Since by assumption $\Aa$ is finite and semi-simple, there is a bimodule decomposition \begin{equation}\label{eq:A1}
    \Aa\cong  \prod_{k=1}^{m} S_0^k\otimes (S_0^k)^\vee,
\end{equation} 
where $\{\Ma^1_0, \dots,S_0^k\}$ is the set of isomorphism classes of irreducible left $\Aa$-modules. Let $S^k = \PhiL(\Ma^k_0)$ be the induced modules, which, by \cref{double dual simplicity} are simple, and by the argument above, are ordinary.

Applying the functor $\Phi$ to \eqref{eq:A1} we obtain \begin{equation}\label{eq:A2}
    \Ac = \Phi(\Aa)= \prod_{k=1}^m \PhiL(S_0^k) \otimes_{\CC} \PhiR(S_0^k)^\vee.
\end{equation} 
Via the involution $\theta$, we may identify the  right indecomposable $\Aa$-modules $(S^k_0)^\vee$ as  left indecomposable $\Aa$-modules $\mbox{}^\theta(S^k_0)^\vee$. Moreover, in \cite[Lemma 1.2]{DGK2} we proved that $\PhiL$ takes indecomposable $\Aa$-modules to indecomposable $V$-modules.  Hence, we obtain a natural identification of indecomposable $V$-modules:
\begin{equation*}\label{eq:A3} \PhiL(\mbox{}^\theta(S^k_0)^\vee)= \PhiR((S^k_0)^\vee) \qquad \text{and} \qquad \PhiL(\mbox{}^\theta(S^k_0)^\vee)_d= \PhiR((S^k_0)^\vee)_{-d},\end{equation*} for every $d \in \NN$. 
Now for each $k\in \{1,\ldots, m\}$,  the contragredient module $(S^k)'$ has degree zero component $(S^k)'_0 = \mbox{}^\theta(S^k_0)^\vee$.  One obtains from this a natural map 
\begin{equation*}    \PhiR((S^k_0)^\vee)=\PhiL(\mbox{}^\theta(S^k_0)^\vee)\to (S^k)',
\end{equation*}
which, by assumption, is an isomorphism. In particular, one has \begin{equation*}\label{eq:A4}\PhiR((S^k_0)^\vee)_{-d}=(S^k_{d})^\vee.\end{equation*}
and so  we can rewrite \eqref{eq:A2} as 
\begin{equation} \label{eq:Acrtl}
    \Ac =\bigoplus_{d\in \ZZ_{\ge 0}}\prod_{k=1}^m S^k_d \otimes_{\CC}(S^k_{d})^\vee.
\end{equation} 
As $S$ and $S'$ are ordinary, we have $\dim_\CC S_d^k$, and $\dim_\CC {S_d^k}^\vee$ are finite. Unraveling the definition of the $\star$-product of $\Ac$ (see the last paragraph of \cite[Remark 3.4.6]{DGK2}), we obtain a ring isomorphism 
\[\Ac_d\cong \prod_{i=1}^m \End_\CC(S^k_d),\] and we conclude that ${\textbf{1}}_d:=\prod_{i=1}^m \text{Id}_{S^k_d}$ is its strong identity element.
\end{proof}

\begin{thm} \label{thm:equivalences}
For $V$ with semi-simple Zhu algebra $\Aa$, the following conditions are equivalent
\begin{enumerate}[label=(\alph*)]
\item \label{thm:a} \label{NTIII} For every simple $\Aa$-module $\Maa$, the generalized Verma module $\PhiL(\Maa)$ is simple and ordinary.
\item \label{thm:b} For every simple $\Aa$-module $\Maa$, the map 
 $\D{\Maa} \coloneqq \PhiL(\cd{\Maa}) \to (\PhiL(\Maa))'$ is an isomorphism.
 \item \label{it:stun} $\Ac_d$ is strongly unital for every $d \in \NN$.
\item \label{it:rtl} $V$ is rational.
\end{enumerate}
\end{thm}

\begin{proof} By assumption, $\Aa$ is semisimple, and hence is finite-dimensional (see \cref{semisimple is finite}).   As already noted in the proof of \cref{prop:we're out here making units}, since $\Aa$ is finite, then necessarily every simple $\Aa$-module is finite-dimensional and so assuming that $\D{\Maa}$ is an isomorphism for every simple $\Aa$-module $\Maa$ is equivalent to $\DD{\Maa}$ being an isomorphism. We, therefore, deduce that \ref{NTIII} and \ref{thm:b} are equivalent by \cref{prop:want modules2}.  That \ref{thm:b} implies \ref{it:stun} follows from \cref{prop:we're out here making units}.

To complete our argument, we will show that \ref{it:stun} implies \ref{it:rtl} and \ref{it:rtl} implies \ref{NTIII}. For the latter, suppose $V$ is rational and let $\Maa$ be a simple $A$-module. As $\PhiL(\Maa)$ is indecomposable and $V$ is rational, it must, in fact, be simple (see also \cite[Section 1.7]{DGT2}). But as $V$ is rational, by definition, all simple modules are ordinary, showing that \ref{NTIII}  holds. Although we will not use it, we note that \cite[Remark 3.4.6]{DGK2} shows that \ref{it:rtl} implies \ref{it:stun}.

Finally, we show that \ref{it:stun} implies \ref{it:rtl}.
In view of \cite[Theorem 4.11]{DongLiMason:HigherZhu}, to show that $V$ is rational, it suffices to show that $\Aa_d$ is finite and semisimple for every $d \in \NN$. By \cite[Theorem 6.0.1 (a)]{DGK2}, since the $\Ac_d$ admit identity elements for all $d$, one has a ring isomorphism $\Aa_d\cong \prod_{e=0}^d\Ac_e$. It, therefore, remains to show that the mode transition algebras $\Ac_d$ are finite and semisimple for every $d \in \NN$. By \cref{bad idempotent ideal}, one has that $\Ba_d$ is an idempotent ideal of $\Ac_d$. As observed in Example~\ref{ss example}, 
the ideal $\Ba_d$ is unital, and therefore by \cref{thm:Acd_Bad}, one has that $\Ac_d$  and $\Ba_d$ are Morita equivalent. 

This equivalence of module categories implies that one module category is semisimple if and only if the other is, and hence, one obtains the corresponding statement for the rings. We also note that by Wedderburn-Artin theory, these rings being finite dimensional $\CC$-algebras is equivalent to their centers being finite products of copies of the complex numbers. But as the centers of these algebras coincide with the centers of their module categories (i.e., the endomorphisms of the corresponding identity functors), it follows that one is finite-dimensional if and only if the other one is.

Therefore, it suffices to show that $\Ba_d$ is finite and semisimple. But this follows from the Wedderburn--Artin Theorem: $\Aa$ is a product of simple finite dimensional algebras. So every two-sided ideal of $\Aa$ is a product of a subset of those simple finite dimensional algebras, and so is semisimple and finite-dimensional. 
\end{proof}

\section{Algebraic consequences of Morita equivalence}\label{Sec:Algebraic}
Here, we prove  \cref{Cor2} and \cref{Cor1},  in which explicit expressions for higher-level Zhu algebras are given in case the $\Ac_d$ are strongly unital for all $d$. These generalize the statement for higher level Zhu algebras for the
Heisenberg VOA, predicted by Addabbo and Barron \cite[Conjecture 8.1]{ABCon},  checked for $d=1$ in \cite[Remark 4.2]{BVY}, and $d=2$ in \cite[Theorem 7.1]{ABCon}, and proved for all $d$ in \cite[Corollary 7.3.1]{DGK2}. 

The proofs of \cref{Cor2} and \cref{Cor1} rely on \cite[Theorem 6.0.1]{DGK2}, where it is shown that if $\Ac_d$ is unital for all $d$, then one has a ring isomorphism 
\begin{equation}\label{eq:Decomp}
    \Aa_d\cong \prod_{e=0}^d\Ac_e.
\end{equation}

 \begin{cor}\label{Cor2} Let $V$ be a rational VOA. Then, for every $d \in \NN$ one has  \[\Aa_d(V)\cong \prod_{j=0}^d \prod_{i=1}^m {\Mat}_{D^i_j}(\CC),\] where 
 $D^i_j={\dim}(S^i_j)$. 
\end{cor}

 \begin{proof}Mode transition algebras $\Ac_d$ for rational VOAs are strongly unital for every $d$ by \cite[Remark 3.4.6]{DGK2}.  One has, by the proof of \cref{prop:we're out here making units}, the decomposition
\[ \Ac_d \cong \prod_{i=1}^m \Hom_\CC(S^i_d,S^i_d),
\] where $\{1,\dots, m\}$ is the set indexing the isomorphism classes of simple $V$-modules and $S^i_d$ is the $d$-th graded piece of $S^i$, namely $S^i=\bigoplus_{j \in \NN} S^i_j$. Combining this with \eqref{eq:Decomp}, one obtains the isomorphism.\end{proof}

We note that the proof of  \cref{Cor1} specializes to give an alternative argument for \cite[Corollary 7.3.1]{DGK2}.   

\begin{cor}\label{Cor1} Let $V$ be a VOA that is simple and generated in degree zero as a module over itself. Let $d$ be a positive integer and suppose that, for all $j \leq d$,  $j$ is non-exceptional for $V$ and $\Ac_j$ is strongly unital. Assume further that $\Aa$ is commutative, connected, Noetherian, and every projective $\Aa$-module is free. Then \[\Ac_j \cong {\Mat}_{D(j)}(\Aa(V)) \quad \text{and thus} \quad \Aa_d(V)\cong \prod_{j=0}^d {\Mat}_{D(j)}(\Aa(V)),\]
 where $D(j)={\dim}_{\CC}(V_j)$.
 \end{cor}

\begin{proof} Since the $\Ac_j=\Ac_j(V)$ admit identity elements for all $j \leq d$, by \cite[Theorem 6.0.3 (a)]{DGK2}, one has that $\Aa_d\cong \prod_{j=0}^d\Ac_j$. Since they are strong identity elements, and we are in the assumptions of \cref{cor:connected-equiv}, the categories of $\Ac_j$-modules are Morita equivalent to the category of $\Aa$-modules for every $j \leq d$. This gives that $\Ac_j\cong {\End}_{\Aa}({\Pa})$, where ${\Pa}$ is a projective generator in the category of $\Aa$-modules \cite[Cor.22.4(c)]{AndersonFuller}.  Our assumption that projective $\Aa$-modules are free gives that ${\Pa}\cong \Aa^{D(j)}$, for some $D(j)$.  It follows that
     \[\Ac_j \cong {\End}_{\Aa}({\Pa}) \cong {\End}_{\Aa}(\Aa^{D(j)})\cong  {\Mat}_{D(j)}(\Aa).\]
We are left to prove that $D(j)=\dim_\CC(V_j)$. To do so, consider an ideal $\Ia \subset \Aa$ such that $\Aa/\Ia \cong \CC$ and note that we can identify $V$ with $\PhiL(\Aa/\Ia)$. Since $\Aa$ is commutative, then $\Ac_j$ is an $\Aa$-module and furthermore, by the description of $\Ac$ given in \eqref{eq:PhiL}, one has
\[\Ac_j\otimes_\Aa \Aa/\Ia \cong \PhiL(\Aa/\Ia)_{j}\otimes_{\Aa/\Ia}\PhiR(\Aa/\Ia)_{-j},\] 
and so 
     \[{\dim}_{\CC}(\Ac_j\otimes_\Aa \Aa/\Ia)=\left({\dim}_{\CC}(\PhiL(\Aa/\Ia))_j \right)^2= \left({\dim}_{\CC}V_j\right)^2.\]
 On the other hand, as $\Aa/\Ia \cong \CC$
     \[{\dim}_{\CC}(\Ac_j\otimes_\Aa \Aa/\Ia)={\dim}_{\CC}({\Mat}_{D(j)}(\Aa)\otimes_\Aa \CC)={\dim}_{\CC}({\Mat}_{D(j)}(\Aa)) = D(j)^2,\]
     which then concludes the proof of the statement.
 \end{proof}

\section{Presentations of mode transition algebras, zig-zag algebras, and Zhu algebras}\label{sec:Computations}
Here, we give explicit presentations for the mode transition algebras, zig-zag algebras, and higher Zhu algebras for several examples,  including the rank-$n$ Heisenberg VOA via  \cref{Cor1} in \cref{sec:Heisenberg}, and an array of rational VOAs via \cref{Cor2} in \cref{Sec:Examples}.
\subsection{Rank-$n$ Heisenberg VOA}\label{sec:Heisenberg}

This section describes all the Zhu algebras of Heisenberg VOAs of arbitrary rank. Let us introduce some combinatorial language, which will be used throughout.

In what follows, by a \emph{partition}, we mean an unordered collection of positive integers with repetition, which we denote by
\[\sigma = \{\ell_1, \ell_2, \ldots, \ell_k\}, \quad \text{ with } \quad \ell_i \in \NN_{\geq 1}.\] 
We say that $\sigma$ is a partition of $m$ if $|\sigma| \coloneq \sum_{i=1}^k \ell_i = m$. We denote by $\PRT_m$ the set of partitions of $m$ and by $\prt_m$ the cardinality of $\PRT_m$. Considering the empty partition, $\PRT_0$ is a singleton, so $\prt_0=1$.

For a positive integer $n$, we define an \emph{$n$-labeled partition} to be the data of a sequence of $n$-many (possibly empty) partitions $\sigma^1,\dots, \sigma^n$, which we denote by
\[\vec{\sigma} = (\sigma^1| \ldots| \sigma^n).\] We say that $\vec{\sigma}$ is a partition of $m$ if $\{|\sigma^1|, \ldots, |\sigma^n|\}$ is a partition of $m$ or, equivalently, if $|\vsigma| \coloneqq \sum_{i=1}^n |\sigma^i|=m$. We let $\PRT^n_m$ denote the set of $n$-labeled partition partitions of $m$ and by $\prt^n_m$ its cardinality.

With this terminology, we can state the main result of this section:

\begin{thm}\label{ThmHeisenberg}
Let $V$ be the rank-$n$ Heisenberg VOA. Then 
\[\Aa_d(V) \cong \prod_{m=0}^d\Mat_{\prt^n_m}(\Aa), \]
where $\Aa = \CC[h_1, \ldots, h_n]$.
\end{thm}
\cref{ThmHeisenberg} generalizes \cite[Corollary 7.3.1]{DGK2}.  Namely, for $\pi$ the rank-$1$ Heisenberg VOA we have 
\begin{equation}\label{oldH} \Ac_j \cong \Mat_{\prt_m}(\Aa) \quad \text{and thus} \quad  \Aa_d(\pi)\cong \prod_{m=0}^d \Mat_{\prt_m}(\Aa).\end{equation} Note that, by Zhu's character formula \cite[Introduction, p. 238]{ZhuMod}, one finds that indeed $\prt_m=\dim(\pi_m)$ and so the decomposition \eqref{oldH} matches with the decomposition one would obtain using \cref{Cor1}. The formula was conjectured by Addabo and Barron, and was shown for $d=1$  in  \cite[Remark 4.2]{BVY}  and $d=2$ by \cite[Theorem 7.1]{addabbo.barron:level2Zhu}.

We will prove \cref{ThmHeisenberg} as a consequence of \cref{Cor1}. To do so, we first describe the universal enveloping algebra of $V$ and then explicitly provide strong identities in $\Ac_d$. 

In what follows, we denote by $\hc$ an abelian Lie algebra over $\CC$ of dimension $n$ with basis $H_1, \dots, H_n$ and with the symmetric bilinear form $(H_i,H_j) \coloneqq \delta_{i,j}$. The associated Heisenberg Lie algebra is \[\whc \coloneqq \hc \otimes \CC[t, t^{-1}] \oplus \CC \kk\] where the Lie bracket is given by
\[ [X f(t), Y g(t)] = (X,Y) \res(g(t)df(t)) \kk,
\] for every $X,Y \in \hc$ and $f,g \in \CC[t,t^{-1}]$, and $\kk$ is a central element. (Here and in what follows, we write $Hf(t)$ instead of $H \otimes f(t)$.)

\label{pp:dimVm} Associated to this Lie algebra, one defines the \emph{rank-$n$ Heisenberg VOA} $V=V(\hc)$, also called the rank-$n$ free boson VOA. This VOA is also identified with the $n$-fold tensor product of the rank-$1$ Heisenberg VOA $\pi$. Since $\dim(\pi_m)=\prt_m$ and the grading of $V$ is induced from the grading of $\pi$, one immediately obtains that $\dim(V_m) = \prt^n_m$. 

The universal enveloping algebra of $V$ can be described as a quotient of a completion of the universal enveloping algebra of $\whc$.
The algebra $\whc$ is naturally graded by the assignments \[\deg(X t^a)=-a \quad \text{and} \quad \deg(\kk)=0,\] so that also the universal enveloping algebra $U(\whc)$ is graded. Denote by $U(\whc)_{\kk=1}$ the quotient of $U(\whc)$ by the identification $\kk=1$, and observe that this is still a graded algebra.   We then define 
\[ \UV= \UV(\hc) \coloneqq \varprojlim \dfrac{U(\whc)_{\kk=1}}{U(\whc)_{\kk=1}  (\hc \otimes t^N \CC[t])},
\] which again is an algebra graded over $\ZZ$, with homogeneous graded piece denoted $\UV_d$. In analogy to \cite[Lemma 4.3.2]{bzf}, one obtains the following identification (see also the comments in \cite[\S 5.1.8]{bzf}).

\begin{lem} \label{lem:Completion} There is an isomorphism of graded algebras between the universal enveloping algebra of $V$ and $\UV$. 
\end{lem}

In particular, every element of $\UV_d$ is represented by (convergent) sums of strings of the form 
\begin{equation} \label{eq:PBW} (X_1 t^{-\ell_1}) \cdot (X_2 t^{-\ell_2}) \cdot \cdots \cdot  (X_r t^{-\ell_r}) \end{equation}
with $X_i \in \hc$ and with $\sum_{i=1}^n \ell_i = d$. Using the standard PBW-style conventions without loss of generality, we may also assume $\ell_1 \geq \ell_2 \geq \cdots \geq \ell_r$. 

Recall that $\Aa_0$ is described as the quotient of $\UV_0$ by elements of $\UV \cdot \UV_{\leq -1}$ of degree zero (see e.g.~\cite{He}). The above description of $\UV_d$ specialized to the case $d=0$ allows us to explicitly describe $\Aa=\Aa_0(V)$ as the polynomial ring in $n$-variables:

\begin{lem} \label{lem:A0V} The map $\CC[h_1, \dots, h_n] \to \Aa$ given by the assignment $h_1 \mapsto H_1$ is an isomorphism of associative algebras. 
\end{lem}

We now show that $\Ac_d$ is strongly unital, so it is convenient to use the language of partitions already introduced. First of all, recall from \cite{DGK2} that 
\[ \Ac_d= \PhiL(\Aa)_d \otimes_\Aa \PhiR(\Aa)_{-d},
\]where 
\[\PhiL(\Aa)_d=\dfrac{\UV_d}{(\UV \cdot \UV_{\leq -1})_d} \Otimes{\UV_0} \Aa \qquad \text{and} \qquad \PhiR(\Aa)_{d}= \Aa \Otimes{\UV_0}  \dfrac{\UV_{-d}}{(\UV_{\geq 1} \cdot \UV)_{-d}}.\]
We now construct the maps
\begin{equation} \label{eq:PRTPhiLd}\PhiR(\Aa)_{-d} \leftarrow \PRT^n_d \rightarrow \PhiL(A)_d, \quad \bar{u}_{\vsigma} \mapsfrom \vsigma \mapsto u_{\vsigma},
\end{equation} which allows us to combinatorially express the elements of $\Ac_d$. Let $\vsigma=(\sigma^1|\dots|\sigma^n) \in \PRT^n_d$, so that the $i$-th partition $\sigma^i$ is given by $\{\ell^i_1, \dots, \ell^i_{k_i}\}$. Associated to $\sigma^i$ we consider the elements 
\[u_{\sigma^i} \coloneqq  H_i t^{-\ell^i_1} \cdots H_i t^{-\ell^i_{k_i}} \quad \text{and} \quad \bar{u}_{\sigma^i} \coloneqq H_i t^{\ell^i_1} \cdots H_i t^{\ell^i_{k_i}}.
\] Note that since terms of negative degree commute with each other (and terms of positive degree commute with each other), the order of the above expressions doesn't matter, so the assignments $\bar{u}_{\sigma^i} \mapsfrom \sigma^i \mapsto u_{\sigma^i}$ are well defined. It follows that \eqref{eq:PRTPhiLd} can be defined by defining the elements
\[ u_{\vsigma} \coloneqq u_{\sigma^1} \cdots u_{\sigma^n} \otimes 1 \in \PhiL(\Aa)_d,
\quad \text{and} \quad \bar{u}_{\vsigma} \coloneqq 1 \otimes \bar{u}_{\sigma^1} \cdots \bar{u}_{\sigma^n} \in \PhiR(\Aa)_{-d}.
\]
Using this notation and \eqref{eq:PBW}, one can show that
\[\PhiL(\Aa)_d = \left\langle u_\sigma \otimes a \mid \sigma \in \PRT^n_d, \ a \in \Aa \right\rangle, \quad \text{and}\quad \PhiR(\Aa)_{-d} = \left\langle a \otimes \bar{u}_{\tau}  \mid \sigma \in \PRT^n_d, \ a \in \Aa \right\rangle
\] It is useful to rescale the elements above by the positive integer  $||\vsigma|| \coloneqq \prod_{i=1}^n ||\sigma^i||$ where $||\sigma^i||$ is defined as in \cite[Example 7.2.2]{DGK2}.

\begin{prop}\label{prop:strong-identity}
$1_d \coloneqq \sum_{\sigma \in \PRT^n_d} \frac{1}{||\vsigma||} u_{\vsigma} \otimes \bar{u}_{\vsigma} \in \Ac_d$ is the strong identity element of $\Ac_d$.
\end{prop}

\begin{proof} By symmetry of the argument, it is enough to show that $1_d$ acts as the identity on $\Ac_{d,0}=\PhiL(\Aa)_d$. We are then left to show that $1_d \star (u_{\vtau} \otimes a) = u_{\vtau} \otimes a$ for every $\vtau \in \PRT^n_d$.
Using the notation as in \cref{Morita}, this is equivalent to showing that 
\[ \star(\bar{u}_{\vsigma} \MAP u_{\vtau}) = \begin{cases} ||\vsigma|| & \text{ if } \vtau=\vsigma\\
0 & \text{ otherwise.} \end{cases}
\]
More concretely, the element  $\star(\bar{u}_{\vsigma} \MAP u_{\vtau})$ coincides with the image of the element
\[u \coloneqq  \bar{u}_{\sigma^1} \cdots \bar{u}_{\sigma^n} \cdot {u}_{\tau^1} \cdots {u}_{\tau^n} \in \UV_0\]
in $\Aa$. Note that since $(H_i,H_j)=0$ if $i \neq j$, then $\bar{u}_{\sigma^i}$ commutes with  $\bar{u}_{\sigma^j}$ and $u_{\tau^j}$. Thus, $u$ is identified with the product of the elements 
\[u_i \coloneqq \bar{u}_{\sigma^i}\cdot u_{\tau^i}
\] in any order. If $\vsigma \neq \vtau$ then there exist a $j \in \{1,\dots,n\}$ such that $\sigma^j \neq \tau^j$. If $|\sigma^j|\neq|\tau^j|$, then $\deg(u_j) \neq 0$, and so, by moving $u_j$ either at the end or at the beginning of the product defining $u$, one shows that $u$ maps to zero in $\Aa$. We can then assume that $|\sigma^j|=|\tau^j|$ and, in particular, that neither of them is the only partition of zero. Since they are different, it means that there exist positive integers $\ell \neq r$ such that we can write $\sigma=\{\sigma',\ell\}$ and $\tau^j=\{\tau', r\}$ for some smaller, possibly empty, partitions $\sigma'$ and $\tau'$. It follows that
\[u_j = \bar{u}_{\sigma'} \cdot H t^{\ell} \cdot Ht^{-r} \cdot u_{\tau'} = \bar{u}_{\sigma'} \cdot Ht^{-r} \cdot H t^{\ell} \cdot {u}_{\tau'}= Ht^{-r} \cdot \bar{u}_{\sigma'} \cdot {u}_{\tau'} \cdot H t^{\ell}, 
\] which again maps to the zero element in $\Aa$. We are then left to show that when $\vsigma=\vtau$, the image of $u$ coincides with $||\vsigma||$. For this purpose, it is enough to show that the image of $u$ in $\Aa$ is equal to the image of the element
\[||\sigma^1|| \left(\bar{u}_{\sigma^2} \cdots \bar{u}_{\sigma^n}\cdot {u}_{\sigma^2} \cdots {u}_{\sigma^n}\right) .
\] This reduces the statement to the rank-$1$ Heisenberg VOA, which follows from \cite[Proposition 7.2.1 and Example 7.2.2]{DGK2}.  \end{proof}

We are ready to complete the proof of \cref{ThmHeisenberg}.

\begin{proof}[Proof of \cref{ThmHeisenberg}] The rank-$n$ Heisenberg VOA $V$ is simple and, by construction, generated in degree zero, and since  \cref{pp:dimVm}, $\dim(V_d)=\prt^n_d$, graded components are nonzero. In particular, every natural number $d$ is non-exceptional for $V$. Moreover,  \cref{lem:A0V} ensures that $\Aa$ is commutative and connected, and that every projective $\Aa$-module is free. Finally, by \cref{prop:strong-identity}, $\Ac_d$ is strongly unital for every $d$. The result thus follows from \cref{Cor1}.
\end{proof}

\subsection{The Virasoro VOA}\label{ExVir} We describe here the degree 1 zig-zag algebra $\Ba_1$ for non-rational Virasoro VOAs. For $c \in\CC$, let $V=M_c= {M_{c,0}}/{<L_{-1}1>}$ be the (not necessarily simple) Virasoro VOA of central charge $c$ (see \cite{WWang}). In \cite[Proposition 8.1.1]{DGK2}, we explicitly realized $\Ac_1(V)$ as an ideal of $\Aa_1$ and showed that it is not (strongly) unital. Here, instead, we explicitly compute $\Bc_1$ and $\Ba_1$.  By \cite{WWang} there is an isomorphism $\Aa_0=\CC[t]$. 

\begin{prop} The $\CC[t]$-algebra $\Bc_1$ is generated by one element $u$ satisfying $u^2=2ut$. The ideal $\Ba_1$ of $\CC[t]$ is isomorphic to $t \CC[t]$. 
\end{prop}

\begin{proof} We refer to \cite[Proof of Proposition 8.1.1]{DGK2} for the notation used throughout.  From \cite[Proof of Proposition 8.1.1]{DGK2} we know that $\Ac_{1,-1}=u_1 \Aa_0 u_{-1}$ and similarly, we obtain that $\Ac_{0,-1}=\Aa_0 u_{-1}$ and $\Ac_{1,0}=u_1 \Aa_0$. It follows that $\Bc_1$ is generated over $\CC[t]$ by $u=u_{-1} \otimes u_{1}$. The product structure is given by
\[ u \MAP u= u_{-1} \otimes u_{1}*u_{-1}*u_{1} = u (u_{-1}*u_{1}),
\] so that to conclude, it is enough to show that the image of $u_{-1}*u_{1}$ in $\Aa=\CC[t]$ is equal to $2t$. By the commutator formula, as in the proof of \cite[Proposition 8.1.1, (a)]{DGK2} we indeed obtain 
\[ u_{-1}*u_{1}=u_{1}*u_{-1} + [u_{-1},u_{1}] = 0+ [u_{-1},u_{1}] = 2t.\]
This shows also that $\Ba_1$ is the ideal of $\CC[t]$ generated by $2t$, or equivalently, generated by $t$.
\end{proof}

\subsection{Lattice VOAs}\label{Sec:Examples}
Let $V_L$ be the VOA associated to an even lattice $L$ of rank $d$ and positive definite quadratic form $Q=q(\, , \,)/2$. Then $V_L$ is rational of conformal dimension $d$, and the set of simple $V_L$-modules is in bijection with $L'/L$, where $L'\in L \otimes \QQ$ denotes the dual lattice of $L$ (with respect to $q$). We denote by $V^\lambda$ the simple module associated with $[\lambda] \in L'/L$ and recall that its conformal dimension is $a_{\lambda}=\frac{1}{2}{\rm{min}_{e\in L}}q(\lambda+e,\lambda+e)$. Thus, to complete the description of $\Aa_d$ using \cref{Cor2}, one only needs to compute the dimensions $D^{\lambda}_j=\dim(V^\lambda_j)$. As demonstrated in \cite{DG}, the integers $D^{\lambda}_j$ can be computed via Zhu's character formula:
\begin{equation*}\label{eq:ZMC2}
\sum_{n\ge 0}^{\infty}{\dim}V^{\lambda}_{n}\, q^n  
=\sum_{n\in\QQ_{\geq 0}}\left(\sum_{\stackrel{n_1,n_2,\ldots,n_d \in \NN}{j\in\QQ_{\geq 0},\sum_{i=1}^d n_i+j=n}} |L^\lambda_j|\prod_{i=1}^d \prt_{n_i} \right)q^{n-a_{\lambda}},\nonumber
\end{equation*}
for $L_j^{\lambda}\coloneqq \{\alpha \in L \ |\   Q(\alpha + \lambda) = j\}$, and where $\prt_n$ denotes the number of partitions of $n$.

For instance, as in \cite{DG}, the coefficient of $q^0$  on the right hand side, which is $D^\lambda_0$,  is equal to the number of ways to write 
$a_\lambda=n_1+n_2+\cdots +n_d+j, \ \mbox{ with }  \ n_i \in \ZZ_{\geq 0}, \ \mbox{ and } j\in\QQ_{\geq 0}$,
and for each such way, the contribution is given by $|L_m^{\lambda}|\prod_i\prt_{n_i}$. 
For $V^0=V_L$ one has $a_0=0$ and so $\dim((V_L)_0)=1$. The dimension of $(V_L)_i$ for $i \geq 1$ depends on the quadratic form $Q$. For instance, one has that $\dim((V_L)_1)=d + N$ where $N$ is the number of elements of $L$ such that $Q(\alpha)=1$.

\begin{ex} For the case from \cite[Example 9.2.1]{DG}: $L=e\ZZ$, with pairing given by $q(e,e)=8$,  simple modules are indexed by $\lambda \in L'/L\cong \ZZ/8\ZZ=\{0, \dots, 7\}$, with conformal dimensions:
\[\begin{array}{c|cccccccc}
    \lambda &  0 & 1 & 2 & 3 & 4 & 5 &6 & 7\\ \hline
     a_\lambda & 0 & \frac{1}{16} & \frac{1}{4} & \frac{9}{16} & 1& \frac{9}{16} & \frac{1}{4}& \frac{1}{16}
\end{array}
\]
When $\lambda \neq 0 \mod 4$, then we have that there is only one way to write $a_\lambda = N+j$ with $N \in \ZZ_{\geq 0}$, and $j \in \QQ_{\geq 0}$, which is to set $N=0$ and $j=a_\lambda$. Thus, for these modules, one obtains $D^\lambda_0= 1$. When $\lambda =4$, instead one obtains $D^4_0=2$ because $|L^4_1|=\{0, -e\}$. \end{ex}

\bibliographystyle{alpha}
\bibliography{biblio}

\end{document}